\documentclass[graybox]{svmult}

\usepackage{mathptmx}       
\usepackage{helvet}         
\usepackage{courier}        
\usepackage{type1cm}        

\usepackage{makeidx}         
\usepackage{graphicx}        
\usepackage{multicol}        
\usepackage[bottom]{footmisc}
\usepackage{amsmath,amssymb,mathrsfs}

\begin{document}

\title*{Einstein metrics and preserved curvature conditions for the Ricci flow}
\author{Simon Brendle}
\institute{Stanford University, 450 Serra Mall, Bldg 380, Stanford CA 94305, U.S.A. \email{brendle@math.stanford.edu}}

\maketitle

\section{Introduction}

In this note, we study Riemannian manifolds $(M,g)$ with the property that $\text{\rm Ric} = \rho \, g$ for some constant $\rho$. A Riemannian manifold with this property is called an Einstein manifold. Einstein manifolds arise naturally as critical points of the normalized Einstein-Hilbert action, and have been studied intensively (see e.g. \cite{Besse}). In particular, it is of interest to classify all Einstein manifolds satisfying a suitable curvature condition. This problem was studied by M.~Berger \cite{Berger}. In 1974, S.~Tachibana \cite{Tachibana} obtained the following important result:

\begin{theorem}[S.~Tachibana]
\label{Tachibana.theorem}
Let $(M,g)$ be a compact Einstein manifold. If $(M,g)$ has positive curvature operator, then $(M,g)$ has constant sectional curvature. Furthermore, if $(M,g)$ has nonnegative curvature operator, then $(M,g)$ is locally symmetric.
\end{theorem} 

In a recent paper \cite{Brendle-Duke}, we proved a substantial generalization of Tachibana's theorem. More precisely, it was shown in \cite{Brendle-Duke} that the assumption that $(M,g)$ has positive curvature operator can be replaced by the weaker condition that $(M,g)$ has positive isotropic curvature:

\begin{theorem}
\label{pic}
Let $(M,g)$ be a compact Einstein manifold of dimension $n \geq 4$. If $(M,g)$ has positive isotropic curvature, then $(M,g)$ has constant sectional curvature. Moreover, if $(M,g)$ has nonnegative isotropic curvature, then $(M,g)$ is locally symmetric.
\end{theorem}

The proof of Theorem \ref{pic} relies on the maximum principle. One of the key ingredients in the proof is the fact that nonnegative isotropic curvature is preserved by the Ricci flow (cf. \cite{Brendle-Schoen}). 

In this note, we show that the first statement in Theorem \ref{pic} can be viewed as a special case of a more general principle. To explain this, we fix an integer $n \geq 4$. We shall denote by $\mathscr{C}_B(\mathbb{R}^n)$ the space of algebraic curvature tensors on $\mathbb{R}^n$. Furthermore, for each $R \in \mathscr{C}_B(\mathbb{R}^n)$, we define an algebraic curvature tensor $Q(R) \in \mathscr{C}_B(\mathbb{R}^n)$ by 
\[Q(R)_{ijkl} = \sum_{p,q=1}^n R_{ijpq} \, R_{klpq} + 2 \sum_{p,q=1}^n (R_{ipkq} \, R_{jplq} - R_{iplq} \, R_{jpkq}).\] 
The term $Q(R)$ arises naturally in the evolution equation of the curvature tensor under the Ricci 
flow (cf. \cite{Hamilton}). The ordinary differential equation $\frac{d}{dt} R = Q(R)$ on $\mathscr{C}_B(\mathbb{R}^n)$ will be referred to as the Hamilton ODE.

We next consider a cone $C \subset \mathscr{C}_B(\mathbb{R}^n)$ with the following properties: 

(i) $C$ is closed, convex, and $O(n)$-invariant.

(ii) $C$ is invariant under the Hamilton ODE $\frac{d}{dt} R = Q(R)$.

(iii) If $R \in C \setminus \{0\}$, then the scalar curvature of $R$ is nonnegative and the Ricci tensor of $R$ is non-zero.

(iv) The curvature tensor $I_{ijkl} = \delta_{ik} \, \delta_{jl} - \delta_{il} \, \delta_{jk}$ lies in the interior of $C$.

We now state the main result of this note:

\begin{theorem}
\label{general.principle}
Let $C \subset \mathscr{C}_B(\mathbb{R}^n)$ be a cone which satisfies the conditions (i)--(iv) above, and let $(M,g)$ be a compact Einstein manifold of dimension $n$. Moreover, suppose that the curvature tensor of $(M,g)$ lies in the interior of the cone $C$ for all points $p \in M$. Then $(M,g)$ has constant sectional curvature.
\end{theorem}

As an example, let us consider the cone 
\[C = \{R \in \mathscr{C}_B(\mathbb{R}^n): \text{\rm $R$ has nonnegative isotropic curvature}\}.\] 
For this choice of $C$, the conditions (i) and (iv) are trivially satisfied. Moreover, it follows from a result of M.~Micallef and M.~Wang (see \cite{Micallef-Wang}, Proposition 2.5) that $C$ satisfies condition (iii) above. Finally, the cone $C$ also satisfies the condition (ii). This was proved independently in \cite{Brendle-Schoen} and \cite{Nguyen}. Therefore, Theorem \ref{pic} may be viewed as a subcase of Theorem \ref{general.principle}.

\section{Proof of Theorem \ref{general.principle}}

The proof of Theorem \ref{general.principle} is similar to the proof of Theorem 16 in \cite{Brendle-Duke}. Let $(M,g)$ be a compact Einstein manifold of dimension $n$ with the property that the curvature tensor of $(M,g)$ lies in the interior of $C$ for all points $p \in M$. If $(M,g)$ is Ricci flat, then the curvature tensor of $(M,g)$ vanishes identically. Hence, it suffices to consider the case that $(M,g)$ has positive Einstein constant. After rescaling the metric if necessary, we may assume that $\text{\rm Ric} = (n-1) \, g$. As in \cite{Brendle-Duke}, we define an algebraic curvature tensor $S$ by 
\begin{equation} 
\label{def.S}
S_{ijkl} = R_{ijkl} - \kappa \, (g_{ik} \, g_{jl} - g_{il} \, g_{jk}), 
\end{equation}
where $\kappa$ is a positive constant. Let $\kappa$ be the largest real number with the property that $S$ lies in the cone $C$ for all points $p \in M$. Since the curvature tensor $R$ lies in the interior of the cone $C$ for all points $p \in M$, we conclude that $\kappa > 0$. On the other hand, the curvature tensor $S$ has nonnegative scalar curvature. From this, we deduce that $\kappa \leq 1$.

\begin{proposition} 
\label{pde.for.S}
The tensor $S$ satisfies 
\[\Delta S + Q(S) = 2(n-1) \, S + 2(n-1) \kappa \, (\kappa - 1) \, I,\] 
where $I_{ijkl} = g_{ik} \, g_{jl} - g_{il} \, g_{jk}$.
\end{proposition} 

\begin{proof} 
The curvature tensor of $(M,g)$ satisfies 
\begin{equation} 
\label{pde.for.R}
\Delta R + Q(R) = 2(n-1) \, R 
\end{equation}
(see \cite{Brendle-Duke}, Proposition 3). Using (\ref{def.S}), we compute 
\begin{align*}
Q(S)_{ijkl} 
&= Q(R)_{ijkl} + 2(n-1) \, \kappa^2 \, (g_{ik} \, g_{jl} - g_{il} \, g_{jk}) \\ 
&- 2\kappa \, (\text{\rm Ric}_{ik} \, g_{jl} - \text{\rm Ric}_{il} \, g_{jk} - \text{\rm Ric}_{jk} \, g_{il} + \text{\rm Ric}_{jl} \, g_{ik}). 
\end{align*} 
Since $\text{\rm Ric} = (n-1) \, g$, it follows that 
\begin{equation} 
\label{QS}
Q(S) = Q(R) + 2(n-1) \kappa \, (\kappa-2) \, I. 
\end{equation}
Combining (\ref{pde.for.R}) and (\ref{QS}), we obtain 
\[\Delta S + Q(S) = 2(n-1) \, R + 2(n-1) \kappa \, (\kappa-2) \, I.\] 
Since $R = S + \kappa I$, the assertion follows. \qed
\end{proof}

In the following, we denote by $T_S C$ the tangent cone to $C$ at $S$.

\begin{proposition}
\label{Laplacian.lies.in.tangent.cone}
At each point $p \in M$, we have $\Delta S \in T_S C$ and $Q(S) \in T_S C$.
\end{proposition}

\begin{proof}
It follows from the definition of $\kappa$ that $S$ lies in the cone $C$ for all points $p \in M$. Hence, the maximum principle implies that $\Delta S \in T_S C$. Moreover, since the cone $C$ is invariant under the Hamilton ODE, we have $Q(S) \in T_S C$. \qed
\end{proof}

\begin{proposition}
\label{kappa}
Suppose that $\kappa < 1$. Then $S$ lies in the interior of the cone $C$ for all points $p \in M$.
\end{proposition}

\begin{proof}
Let us fix a point $p \in M$. By Proposition \ref{Laplacian.lies.in.tangent.cone}, we have $\Delta S \in T_S C$ and $Q(S) \in T_S C$. Furthermore, we have $-S \in T_S C$ since $C$ is a cone. Putting these facts together, we obtain 
\[\Delta S + Q(S) - 2(n-1) \, S \in T_S C.\] 
Using Proposition \ref{pde.for.S}, we conclude that 
\[2(n-1) \kappa \, (\kappa - 1) \, I \in T_S C.\] 
Since $0 < \kappa < 1$, it follows that $-2I \in T_S C$. On the other hand, $I$ lies in the interior of the tangent cone $T_S C$. Hence, the sum $-2I + I = -I$ lies in the interior of the tangent cone $T_S C$. By Proposition 5.4 in \cite{Brendle-book}, there exists a real number $\varepsilon > 0$ such that $S - \varepsilon I \in C$. Therefore, $S$ lies in the interior of the cone $C$, as claimed. \qed
\end{proof}

\begin{proposition}
The algebraic curvature tensor $S$ defined in (\ref{def.S}) vanishes identically.
\end{proposition}

\begin{proof}
By definition of $\kappa$, there exists a point $p_0 \in M$ such that $S \in \partial C$ at $p_0$. Hence, it follows from Proposition \ref{kappa} that $\kappa = 1$. Consequently, the Ricci tensor of $S$ vanishes identically. Since $S \in C$ for all points $p \in M$, we conclude that $S$ vanishes identically. \qed
\end{proof}

Since $S$ vanishes identically, the manifold $(M,g)$ has constant sectional curvature. This completes the proof of Theorem \ref{general.principle}.

\end{document}